\documentclass[11pt,fleqn]{amsart}

\usepackage{amsmath,amssymb,latexsym}
\usepackage[mathscr]{eucal}
\usepackage{graphicx}
\usepackage[pdf]{pstricks}

\usepackage{cite}

\theoremstyle{plain}
\newtheorem{theorem}{Theorem}
\newtheorem{lemma}{Lemma}

\numberwithin{equation}{section}

\def\mydate{\number\year-\ifnum\month<10{0}\fi\number\month-\ifnum\day<10{0}\fi\number\day}

\newcommand{\dy}{\partial}
\newcommand{\ddt}[1]{\frac{\mathrm{d}{#1}}{\mathrm{d}{t}}}

\newcommand{\sfrac}[2]{{\textstyle\frac{#1}{#2}}}

\newcommand{\eps}{\varepsilon}

\newcommand{\ncdot}{\!\cdot\!}

\newcommand{\gb}{\nabla}

\newcommand{\cpoi}{c_0}
\newcommand{\Dom}{\Omega}
\newcommand{\ktil}{\tilde K}

\begin{document}

\title[Timestepping Schemes for 3d NSE]%
{Timestepping Schemes for\\
the 3d Navier--Stokes Equations:\\
Small Solutions and Short Times}

\author[Hong]{Youngjoon Hong}
\email{hongy@indiana.edu}
\urladdr{http://pages.iu.edu/~hongy}
\address[YJH]{Mathematics Department, Indiana University, Bloomington, IN~47405--7106, United States}

\author[Wirosoetisno]{Djoko Wirosoetisno}
\email{djoko.wirosoetisno@durham.ac.uk}
\urladdr{http://www.maths.dur.ac.uk/\~{}dma0dw}
\address[DW]{Mathematical Sciences, Durham University, Durham\ \ DH1~3LE, United Kingdom}

\thanks{%
This work was supported in part by NSF Grant DMS 1206438
and by the Research Fund of Indiana University.
}


\begin{abstract}
It is well known that the solution of the 3d Navier--Stokes equations
remains bounded if the initial data and the forcing are sufficiently
small relative to the viscosity,
and for a finite time given any bounded initial data.
In this article, we consider two temporal discretisations (semi-implicit
and fully implicit) of the 3d Navier--Stokes equations in a periodic domain
and prove that their solutions remain bounded in $H^1$ subject to
essentially the same smallness conditions (on initial data, forcing or time)
as the continuous system and to suitable timestep restrictions.
\end{abstract}

\maketitle


\section{Introduction}\label{s:intro}

\nocite{geveci:89}
\nocite{marion-temam:98}
\nocite{shen:dugl:90}
\nocite{temam:nse}
\nocite{tone-dw:dugl}
\nocite{gtw3:dodu}
\nocite{wxm:12}
\nocite{constantin-foias:nse}

Much work has been done on the stability and convergence of various
timestepping schemes for the Navier--Stokes equations in two space
dimensions (2d NSE).
The stability of Euler schemes for 2d NSE has been treated in, e.g.,
\cite{geveci:89,shen:dugl:90,ju:02,tone-dw:dugl}, and more recently
extended to higher-order schemes in \cite{wxm:12,gtw3:dodu}.
Given sufficient boundedness of the numerical solution, convergence
can usually be established using now-standard techniques
(cf., e.g., \cite{marion-temam:98}).

In three dimensions, boundedness of the solution for a finite time
(depending on the initial data) follows essentially from the
Cauchy--Kovalevskaya theorem.
It is also well known that, if both the initial data and the forcing
are sufficiently small (relative to the viscosity), the solution will
be globally bounded.
For more background on the NSE, see, e.g.,
\cite{constantin-foias:nse,temam:nse}.

In this article we consider temporal discretisations of the 3d NSE using
the semi-implicit \eqref{q:semi} and fully implicit \eqref{q:full} schemes,
and following ideas from \cite{tone-dw:dugl} prove discrete analogues of
the short-time and small-data boundedness of the continuous-time case.
As in the earlier works cited above, we do not consider spatial
discretisations, giving the advantage that our results will be
free of Courant--Friedrichs--Lewy-type constraints,
although some smallness of the timestep may be required.

\medskip
We consider the Navier--Stokes equations in $\Dom=(0,2\pi)^3$
with periodic boundary conditions,
\begin{equation}\label{q:dudt}\begin{aligned}
   &\dy_t u + (u\ncdot\gb) u + \gb p = \nu\Delta u + f,\\
   &\gb\ncdot u = 0,
\end{aligned}\end{equation}
plus the initial data $u(0)=u_0$.
With no loss of generality, we assume
that $\gb\ncdot f=0$,
and that the integrals of $f$ and $u_0$ vanish over $\Dom$.
The last assumption implies that $u=u(t)$, whenever it is well-defined for
$t\ge0$, also has vanishing integral over $\Dom$, giving us the Poincar{\'e}
inequality
\begin{equation}\label{q:cpoi}
   |u|_{L^2}^2 \le \cpoi(\Dom) |\gb u|_{L^2}^2.
\end{equation}
For notational convenience, we redefine $\cpoi$ to give also the bound
\begin{equation}
   |\gb u|_{L^2}^2 \le \cpoi |\Delta u|_{L^2}^2.
\end{equation}

In order to facilitate comparison with the numerical solutions,
in the rest of this section we briefly review the boundedness of
solutions of the 3d NSE,
both in $L^2$ and in $H^1$ for the two cases (small data and short time).

Multiplying \eqref{q:dudt} by $u$ in $L^2(\Dom)$,
integrating by parts and using the fact that $(u\ncdot\gb u,u)=0$,
we find
\begin{equation} \label{e:1.4}
   \frac12\ddt{\;} |u|^2 + \nu|\gb u|^2 = (f,u).
\end{equation}
Here and henceforth, unadorned norm $|\cdot|$ and inner product $(\cdot,\cdot)$
are taken to be $L^2$.
Bounding the rhs by the Cauchy--Schwarz inequality and using the Poincar{\'e} inequality,
\eqref{e:1.4} becomes
\begin{equation}
   \ddt{\;}|u|^2 + \frac\nu\cpoi |u|^2 \le \frac1\nu |f|^2_{L^{\infty}(H^{-1})},
\end{equation}
where $|f|_{L^{\infty}(H^{-1})}:=\sup_{t\ge0}|f(t)|_{H^{-1}}$.
Integrating, we find the uniform $L^2$ bound
\begin{equation}\label{q:K0}
   |u(t)|^2 \le |u(0)|^2 + (\cpoi/\nu^2) |f|^2_{L^{\infty}(H^{-1})}
	=: K_0(u_0,f;\nu,\Dom).
\end{equation}


\subsection{$H^1$ estimate for small solutions}

Now multiplying \eqref{q:dudt} by $-\Delta u$ in $L^2(\Dom)$ and integrating
by parts, we find
\begin{equation}
   \frac12\ddt{\;}|\gb u|^2 + \nu|\Delta u|^2 = (u\ncdot\gb u,\Delta u) - (f,\Delta u).
\end{equation}
Bounding the nonlinear term using the Sobolev inequality,
\begin{equation}\label{q:c1}
   \bigl|(u\ncdot\gb u,\Delta u)\bigr|
	\le |u|_{L^3}^{}|\gb u|_{L^6}^{}|\Delta u|_{L^2}^{}
	\le \frac{c_1}2|u|_{L^3}^{}|\Delta u|_{L^2}^2,
\end{equation}
and the forcing term in the obvious fashion, we arrive at
\begin{equation}
   \ddt{\;}|\gb u|^2 + (3\nu/2 - c_1|u|_{L^3}^{})|\Delta u|^2 \le \frac2\nu |f|^2.
\end{equation}
Assuming for now that
\begin{equation}\label{q:ul3}
   |u(t)|_{L^3}^{} \le \nu/(2c_1)
	\qquad\textrm{for all }t\ge0,
\end{equation}
we can integrate the differential inequality to obtain
\begin{equation}\label{q:K1}
   |\gb u(t)|^2 \le |\gb u(0)|^2 + (2\cpoi/\nu^2) |f|_{L^{\infty}(L^2)}^2
	=: K_1(u_0,f;\nu,\Dom),
\end{equation}
where $|f|_{L^{\infty}(L^2)} := \sup_{t \geq 0}|f(t)|_{L^2}$.
Using the Sobolev inequality
$|u|_{L^3}^2 \le c\, |u|_{H^{1/2}}^2 \le c\,|u|\,|\gb u|$,
a sufficient condition for \eqref{q:ul3} is
\begin{equation}\label{q:K0K1}
   K_0 K_1 = \bigl(|u_0|^2 + \cpoi\, |f|^2_{L^{\infty}(H^{-1})} \bigr)
	\Bigl(|\gb u_0|^2 + \frac{2\cpoi}{\nu^2}|f|_{L^{\infty}(L^2)}^2\Bigr)
	\le c_2(\Dom)\,\nu^4.
\end{equation}
It therefore follows that whenever this holds, the 3d NSE has a global
solution bounded by \eqref{q:K0} and \eqref{q:K1}.
It will be convenient to use the Poincar{\'e} inequality to derive
a slightly stronger condition that implies \eqref{q:K0K1},
\begin{equation}\label{q:K1K1}
   K_1 = \bigl(|\gb u_0|^2 + 2\cpoi\,|f|_{L^{\infty}(L^2)}^2/\nu^2\bigr) \le c_3(\Dom)\nu^2,
\end{equation}
with $c_3=\sqrt{c_2/\cpoi}$.


\subsection{$H^1$ estimate for short times}

Multiplying \eqref{q:dudt} by $-\Delta u$ in $L^2(\Dom)$
and integrating  by parts, we find
\begin{equation}
    \frac12\ddt{\;}|\gb u|^2 + \nu|\Delta u|^2 = (u\ncdot\gb u,\Delta u) - (f,\Delta u).
\end{equation}
Bounding the nonlinear term using the Sobolev and interpolation inequalities,
\begin{equation}\label{q:c4}\begin{aligned}
    |(u \cdot \nabla u, \Delta u)|
	&\le |u|_{L^6}^{} |\nabla u|_{L^3}^{} |\Delta u|_{L^2}^{}
	& &\le c\,|\gb u|\,|\gb u|_{H^{1/2}}^{}|\Delta u|\\
	&\le c\,|\gb u|^{3/2}|\Delta u|^{3/2}
	& &\le \frac{c_4}{2\nu^3} |\nabla u|^6 + \frac{\nu}{2} |\Delta u|^2,
\end{aligned}\end{equation}
and using the Cauchy--Schwarz inequality for the last term, this gives
\begin{equation}\label{q:dh1dt}
   \ddt{\;}|\nabla u|^2\le \frac{c_4}{\nu^3} |\nabla u|^6 + \frac{1}{\nu} |f|_{L^{\infty}(L^2)}^2.
\end{equation}
This implies
\begin{equation}\label{q:dzdt}
   \ddt{\;}(|\gb u|^2 + F)
	\le \frac{c_4}{\nu^3}(|\gb u|^2 + F)^3,
\end{equation}
where $F:=\bigl(\nu^2|f|_{L^{\infty}(L^2)}^2/c_4\bigr)^{1/3}$.
Writing $z(t):=|\gb u(t)|^2 + F$ and integrating, we find
\begin{equation}\label{q:zt}
   \frac{z(t)^2}{z(0)^2} \le\frac{1}{1-2tc_4z(0)^2/\nu^3},
\end{equation}
as long as $t<\nu^3/(2c_4z(0)^2)$.
It is clear from this that our solution will remain bounded, say,
$z(t)^2\le 2z(0)^2$, for $0\le t\le \nu^3/(4c_4 z(0)^2)$.


\section{Semi-implicit Scheme}\label{s:semi}

Given a fixed $k>0$, we discretise \eqref{q:dudt} in time using the semi-implicit Euler scheme
\begin{equation}\label{q:semi}
   \frac{u^n-u^{n-1}}k + u^{n-1}\ncdot\gb u^n + \gb p = \nu\Delta u^n + f^n,
\end{equation}
with $u^0=u_0$.
For 2d NSE, this scheme was proved in \cite{ju:02} to be globally stable in $H^1$.
For 3d NSE, its stability mirrors that (which is known) of the continuous system,
subject to relatively mild timestep restrictions:

\begin{theorem}\label{t:semi}
For small solutions, let the initial data $u_0 \in H^1$, the forcing $f$
and the time\-step $k$ satisfy
\begin{equation}\label{q:K0K1s}
   ( K_0 + k|f|^2_{L^{\infty}(H^{-1})} /\nu) (K_1 + 2k|f|_{L^{\infty}(L^2)}^2/\nu) \leq c_2(\Dom) \nu^4,
\end{equation}
where $K_0(u_0,f)$ and $K_1(u_0,f)$ are as in the continuous case,
\eqref{q:K0} and \eqref{q:K1}.
Then $u^n$ is bounded in $H^1$ as follows,
\begin{equation}\label{q:bdh1s}
   |\gb u^n|^2 \le K_1 + (2k/\nu)|f|_{L^{\infty}(L^2)}^2
   \qquad\textrm{for all }n\ge0.
\end{equation}
For short times, assuming the timestep restriction \eqref{q:dtf5},
we have
\begin{equation}
   |\gb u^n|^2 \le 2\,|\gb u^0|^2 + F,
\end{equation}
for all $n$ such that $nk=t_n\le \nu^3/\bigl(8c_4(|\gb u^0|^2+F)^2\bigr)$.
\end{theorem}

We note a few facts that will be useful later.
First, for any $a$ and $b\in L^2$,
\begin{equation}\label{q:id0}
   2(a-b,a) = |a|^2 - |b|^2 + |a-b|^2.
\end{equation}
Next, for $b>0$ and positive sequences $(x_n)$ and $(r_n)$,
\begin{equation}\label{q:dgrone}
   (1+b)\,x_n \le x_{n-1} + r_{n-1}
   \qquad\Rightarrow\qquad
   x_n \le (1+b)^{-n} x_0 + \frac{1+b}b {\textstyle\max_j}\, r_j.
\end{equation}

\begin{proof}
The $L^2$ bound works out essentially as in the continuous case:
multiplying $\eqref{q:semi}_1$ by $2ku^n$, using \eqref{q:id0} and
noting that $(u^{n-1}\ncdot\gb u^n,u^n)=0$, we find
\begin{equation}
   |u^n|^2 + |u^n-u^{n-1}|^2 + 2\nu k |\gb u^n|^2 = |u^{n-1}|^2 + 2k (f^n,u^n).
\end{equation}
Bounding the forcing term using Cauchy--Schwarz and Poincar{\'e},
this implies
\begin{equation}
   (1 + \nu k/\cpoi) |u^n|^2 \le |u^{n-1}|^2 + k|f^n|_{H^{-1}}^2/\nu.
\end{equation}
Integrating this using \eqref{q:dgrone}, we find for all $n\in\{1,2,\cdots\}$,
\begin{equation}\label{q:bdl2s}\begin{aligned}
   |u^n|^2 &\le |u^0|^2 + \frac{\cpoi+\nu k}{\nu^2} |f|^2_{L^{\infty}(H^{-1})}\\
	&= K_0(u_0,f;\nu,\Dom) + (k/\nu) |f|^2_{L^{\infty}(H^{-1})},
\end{aligned}\end{equation}
where $K_0$, and $K_1$ below, are as in the continuous case.
We note that this bound tends to $K_0$ as $k\to0$.

\medskip
We now turn to stability in $H^1$ for small solutions.
Multiplying \eqref{q:semi} by $-2k \Delta u^n$ and using \eqref{q:id0},
we find
\begin{equation}\label{q:semi-h1}\begin{aligned}
    |\nabla u^n|^2 &+ |\nabla(u^n - u^{n-1})| + 2 \nu k |\Delta u^n|^2  \\
    &= |\nabla u^{n-1}|^2 - 2k(f^n,\Delta u^n) + 2k(u^{n-1} \cdot \nabla u^n, \Delta u^n).
\end{aligned}\end{equation}
Bounding the nonlinear term using the Sobolev inequality,
\begin{equation}
   \bigl|(u^{n-1}\ncdot\gb u^n,\Delta u^n)\bigr|
	\le |u^{n-1}|_{L^3}^{}|\gb u^n|_{L^6}^{}|\Delta u^n|_{L^2}^{}
	\le c_1|u^{n-1}|_{L^3}^{}|\Delta u^n|_{L^2}^2,
\end{equation}
and using the Cauchy--Schwarz inequality for the forcing, \eqref{q:semi-h1} implies
\begin{equation}\label{q:semi0}
    |\nabla u^n|^2 + (3 \nu/2 - c_1 |u^{n-1}|_{L^3})k |\Delta u^n|^2
	\le |\nabla u^{n-1}|^2 + (2k/\nu)|f^n|^2.
\end{equation}
If we now assume that
\begin{equation}\label{q:hyps}
   |u^{n-1}|_{L^3} \le \nu/(2c_1),
\end{equation}
we deduce from \eqref{q:semi0}
\begin{equation}
    (1+\nu k/\cpoi) |\nabla u^n|^2 \leq |\nabla u^{n-1}|^2 + 2k |f^n|^2/\nu.
\end{equation}
As long as \eqref{q:hyps} holds, we can integrate this using \eqref{q:dgrone}
to get the bound
\begin{equation}\begin{aligned}
    |\nabla u^n|^2
    &\le |\nabla u^0|^2 +  \frac{2(\cpoi + \nu k)}{\nu^2} |f|_{L^{\infty}(L^2)}^2\\
    &\le K_1 (u_0,f;\nu,\Dom) + (2k/\nu) |f|_{L^{\infty}(L^2)}^2,
\end{aligned}\end{equation}
which proves \eqref{q:bdh1s}.
As in the continuous case, we now use Sobolev and interpolation
inequalities to bound
\begin{equation}
   |u^{n-1}|_{L^3}^2 \le c\, |u^{n-1}|_{H^{1/2}}^2
	\le c\,|u^{n-1}|\,|\gb u^{n-1}|.
\end{equation}
The timestep restriction \eqref{q:K0K1s} then becomes a sufficient condition
for \eqref{q:hyps}.
More explicitly, since \eqref{q:K0K1s} holds at $n=0$,
\eqref{q:bdl2s} and \eqref{q:bdh1s}
imply that it will hold at $n=1$ and, by induction, for all $n\in\{2,\cdots\}$,
i.e.\ the solution of the scheme \eqref{q:semi} is bounded
uniformly in (discrete) time subject to \eqref{q:K0K1s}.
Comparing to \eqref{q:K0K1}, we note that this condition also depends
on the timestep $k$ in addition to $u_0$ and $f$.
This timestep restriction is however relatively mild compared to that
for the fully implicit scheme in \S\ref{s:full} below.

\medskip
For short-time $H^1$ stability, we bound the nonlinear term
in \eqref{q:semi-h1} as in \eqref{q:c4},
\begin{equation}
   |(u^{n-1}\cdot\gb u^n,\Delta u^n)|
	\le \frac{\nu}2\,|\Delta u^n|^2 + \frac{c_4}{2\nu^3}|\gb u^{n-1}|^4|\gb u^n|^2,
\end{equation}
to give
\begin{equation}
   |\gb u^n|^2 \le |\gb u^{n-1}|^2 + \frac{c_4k}{\nu^3}|\gb u^{n-1}|^4|\gb u^n|^2
	+ \frac{k}{\nu}|f|_{L^{\infty}(L^2)}^2 .
\end{equation}
We rewrite this as
\begin{equation}\label{q:aus00}\begin{aligned}
   \Bigl(1 &- \frac{c_4k}{\nu^3}|\gb u^{n-1}|^4\Bigr)|\gb u^n|^2\\
	&\quad\le \Bigl(1 - \frac{c_4k}{\nu^3}|\gb u^{n-1}|^4\Bigr)|\gb u^{n-1}|^2
		+ \frac{c_4k}{\nu^3}|\gb u^{n-1}|^6 + \frac{k}{\nu}|f|_{L^{\infty}(L^2)}^2.
\end{aligned}\end{equation}
Since we are interested in short times, we assume that
$|\gb u^{n-1}|^2 \le 2|\gb u^0|^2$ for all relevant $n$ and demand
that $k$ satisfy
\begin{equation}\label{q:dtf5}
   k \le \frac{\nu^3}{2c_4(2|\gb u^0|^2+F)^2},
\end{equation}
where $F>0$ is that in \eqref{q:dzdt}.
This implies that the brackets in \eqref{q:aus00} are ${}\ge\sfrac12$;
we have added the extra $F$ for later use.
With this assumption, \eqref{q:aus00} implies
\begin{equation}\label{q:aus01}
   \frac{|\gb u^n|^2 - |\gb u^{n-1}|^2}{k} \le \frac{2c_4}{\nu^3}|\gb u^{n-1}|^6 + \frac{2}{\nu}|f|_{L^{\infty}(L^2)}^2.
\end{equation}
Unlike its continuous-time analogue \eqref{q:dh1dt}, this difference
inequality implies $|\gb u^n|<\infty$ for all $n$, although for sufficiently
large time $nk$ it (i.e.\ the bound) grows without bound as $k\to0$.
This is a well-known pitfall in discretising differential equations in time.
To obtain a finite-time bound on $|\gb u^n|$, we proceed in analogy
with \eqref{q:dzdt} and define
\begin{equation}
   z_n := |\gb u^n|^2 + F.
\end{equation}
We then get from \eqref{q:aus01}
\begin{equation}
   \frac{z_n-z_{n-1}}{k} \le \frac{2c_4}{\nu^3}z_{n-1}^3
	=: g(z_{n-1}).
\end{equation}
Observe that $g(\zeta)>g(\hat\zeta)$ whenever $\zeta>\hat\zeta$,
that is, $g\ge0$ is strictly monotone increasing function.

Now let $\zeta_n$ be the positive solution of
the difference {\em equation\/},
\begin{equation}\label{q:dd0}
   \frac{\zeta_n-\zeta_{n-1}}{k} = g(\zeta_{n-1}),
\end{equation}
and observe that $\zeta_n \geq 0$ if $\zeta_{n-1} \geq 0$.
Denoting $t_n:=nk$, we claim that $\zeta_n\le \zeta(t_n)$ where $\zeta(\cdot)$
is the solution of the {\em differential\/} equation
\begin{equation}
   \ddt{\zeta} = g(\zeta)
   \qquad\textrm{with }\zeta(t_{n-1})=\zeta_{n-1}.
\end{equation}
To show this, we first note that $\zeta(t)$ is non-decreasing since $g\ge0$.
Then
\begin{equation}\label{q:dd1}
   \zeta(t_n) - \zeta(t_{n-1}) = \int_{t_{n-1}}^{t_n} g(\zeta(t)) \;\mathrm{d}t
	\ge \int_{t_{n-1}}^{t_n} g(\zeta(t_{n-1})) \;\mathrm{d}t
	= kg(\zeta_{n-1}),
\end{equation}
proving our claim.
By induction, taking $\zeta(0)=\zeta_0>0$ instead of the initial data
in \eqref{q:dd0}, we then have
$\zeta_n\le \zeta(t_n)$ for all $n\in\{1,2,\cdots\}$.
Comparing with the continuous case \eqref{q:dzdt}--\eqref{q:zt},
we conclude that $\zeta_n\le\zeta(t_n)\le 2\zeta(0)=2\zeta_0$ for
$nk=t_n\le \nu^3/(8c_4\zeta_0^2)$.

Taking $\zeta_0=z_0$, clearly $z_n\le\zeta_n$ for all $n\ge0$.
We therefore have
\begin{equation}
   |\gb u^n|^2 \le 2\,|\gb u^0|^2 + F,
\end{equation}
for all $n$ such that $nk = t_n \le \nu^3/\bigl(8c_4(|\gb u^0|^2 + F)^2\bigr)$,
which is half as long as the bound in the continuous case.
\end{proof}


\section{Fully Implicit Scheme}\label{s:full}

We now consider the fully implicit Euler scheme
\begin{equation}\label{q:full}
   \frac{u^n-u^{n-1}}k + u^n\ncdot\gb u^n + \gb p = \nu\Delta u^n + f^n,
\end{equation}
with $\gb\cdot u^n=0$ for all $n$ and $u^0=u_0$.
In two space dimensions, uniform boundedness in $H^1$ for this scheme
was proved in \cite{tone-dw:dugl}, whose ideas we borrow below.

The $L^2$ bound obtains as before:
multiplying $\eqref{q:full}_1$ by $2ku^n$ and using \eqref{q:id0},
\begin{equation}
   |u^n|^2 + |u^n-u^{n-1}|^2 + 2\nu k |\gb u^n|^2 = |u^{n-1}|^2 + 2k (f^n,u^n).
\end{equation}
Bounding the forcing term in the obvious manner and using Poincar{\'e},
this implies
\begin{equation}
   (1 + \nu k/\cpoi) |u^n|^2 \le |u^{n-1}|^2 + k |f^n|_{H^{-1}}^2/\nu.
\end{equation}
Integrating this using \eqref{q:dgrone}, we find for all $n\in\{1,2,\cdots\}$,
\begin{equation}\label{q:bdfl2}\begin{aligned}
   |u^n|^2 &\le |u^0|^2 + \frac{\cpoi+\nu k}{\nu^2} |f|^2_{L^{\infty}(H^{-1})}\\
	&= K_0(u_0,f;\nu,\Dom) + (k/\nu) |f|^2_{L^{\infty}(H^{-1})}.
\end{aligned}\end{equation}
As before, this bound tends to $K_0$ as $k\to0$.
For later use, we define
\begin{equation}
   \ktil_0(u_0,f;\nu,\Dom) := |u_0|^2 + \frac{2\cpoi}{\nu^2} |f|^2_{L^{\infty}(H^{-1})}.
\end{equation}

The central ingredient for our main results is the following local-in-time
estimate:

\begin{lemma}\label{t:lem}
We assume the $L^2$ uniform bound \eqref{q:bdfl2} and that $u^{n-1}\in H^1$.
Assuming further the timestep restrictions
\begin{align}
   &K^{(n-1)} \le \frac12\Bigl(\frac{\nu^3}{3c_4k}\Bigr)^{1/2}\label{q:dtfx},\\
   &\Bigl(1 + \frac{{ c_5}}{\nu^4} \ktil_0K^{(n-1)}\Bigr)K^{(n-1)} + |f|^2_{L^{\infty}(H^{-1})}/\nu^2 \le \Bigl(\frac{\nu^3}{{ 12c_4k}}\Bigr)^{1/2},\label{q:dtfy}
\end{align}
where $K^{(n-1)}:=|\gb u^{n-1}|^2 + ({10} c_0/\nu)|f|_{L^{\infty}(L^2)}^2$,
then the solution $u^n$ of \eqref{q:full} is bounded as $|\gb u^n|^2\le y_1$
where $y_1$ is the smallest positive root of the cubic equation \eqref{q:un6}.
\end{lemma}

The crucial point which is not immediately obvious is that
$y_1=|\gb u^{n-1}|^2 + {\sf O}(k)$ for small $k$.
By estimating the ${\sf O}(k)$ more carefully, we obtain our main results:

\begin{theorem}\label{t:full}
For short times, let the timestep $k$ satisfy \eqref{q:dtfx1}, \eqref{q:dtfy1}
and \eqref{q:dtfz} below.
Then there exists a $t_f^*$ such that, as long as $0\le nk\le t_f^*$ we have
\begin{equation}
   |\gb u^n|^2 \le 2\,|\gb u_0|^2 + \,(\nu^2|f|_{L^{\infty}(L^2)}^2/{c_4})^{1/3}.
\end{equation}
For small solutions, let $u_0$ and $f$ be such that
\begin{equation}\label{q:hypf}
   |\gb u_0|^2 + \frac{2\cpoi}{\nu^2}|f|_{L^{\infty}(L^2)}^2
	\le \frac{\nu^2}{{2 \sqrt{\cpoi c_4}}},
\end{equation}
and let the timestep $k$ satisfy \eqref{q:dtf0}--\eqref{q:dtfb} below.
Then $u^n$ is bounded as
\begin{equation}\label{q:bdfh1}
   |\gb u^n|^2
	\le \ktil_1(u_0,f;\nu,\Dom)
	:= |\gb u_0|^2 + \frac{{10 }\cpoi}{\nu^2}|f|_{L^{\infty}(L^2)}^2,
\end{equation}
for all $n\in\{0,1,\cdots\}$.
\end{theorem}

We note that, up to constants depending only on the domain $\Dom$,
both the smallness condition \eqref{q:hypf} and the bound \eqref{q:bdfh1}
are the same as those in the continuous case, \eqref{q:K1K1} and \eqref{q:K1}.
Also, the time bound $t_f^*$ is essentially that in the continuous case
(smaller by a factor of {$\sfrac12$} which can be improved to $1-\eps$
with some work and more restriction on $k$).

\begin{proof}[Proof of Lemma~\ref{t:lem}]
Multiplying \eqref{q:full} by $-2k\Delta u^n$, we have
\begin{equation}\begin{aligned}
   |\gb u^n|^2 &+ |\gb(u^n-u^{n-1})|^2 + 2\nu k |\Delta u^n|^2\\
	&= |\gb u^{n-1}|^2 + 2k (u^n\ncdot\gb u^n,\Delta u^n) - 2k (f^n,\Delta u^n).
\end{aligned}\end{equation}
Bounding the nonlinear term as we did in \eqref{q:c4},
\begin{equation}
   2k\bigl|(u^n\ncdot\gb u^n,\Delta u^n)\bigr|
	\le \frac{{c_4k}}{\nu^3}|\gb u^n|^6 + \nu k |\Delta u^n|^2,
\end{equation}
we find
\begin{align}
   &0 \le \frac{{c_4k}}{\nu^3} |\gb u^n|^6 - |\gb u^n|^2 - {\frac{\nu k}{2}} |\Delta u^n|^2
	+ |\gb u^{n-1}|^2 + \frac{2k}\nu |f^n|^2 \notag\\
   \Rightarrow\quad
   &0 \le \frac{{c_4k}}{\nu^3} |\gb u^n|^6 - \Bigl(1 + \frac{\nu k}{{2\cpoi}} \Bigr)|\gb u^n|^2 + |\gb u^{n-1}|^2 + \frac{2k}\nu |f|_{L^{\infty}(L^2)}^2.\label{q:un6}
\end{align}
Let $y:=|\gb u^n|^2$, $x:=|\gb u^{n-1}|^2 + 2k|f|_{L^{\infty}(L^2)}^2/\nu$ and
\begin{equation}\label{q:G}
   G(y;x) := (c_4k/{\nu^3}) y^3 - (1+\nu k/(2 \cpoi)) y + x.
\end{equation}
We write $G(y)$ instead of $G(y;x)$ when there is no risk of confusion.
We are of course interested in the solution set of $G(y)\ge0$.

\begin{figure}
\begin{center}
\scalebox{1} 
{\begin{pspicture}(0,-2.5)(11.1,2.2) 
\psline[linewidth=0.03cm,arrowsize=0.05cm 2.0,arrowlength=1.4,arrowinset=0.4]{<-}(3.72,2.0)(3.72,-2.6) 
\psline[linewidth=0.03cm,arrowsize=0.05cm 2.0,arrowlength=1.4,arrowinset=0.4]{->}(0.1,-1.2)(11.08,-1.2) 
\usefont{T1}{ppl}{m}{n}
\rput(11.0,-0.9){$y$}
\psbezier[linewidth=0.04cm](0.5,-2.0834374)(0.8644867,-1.6677852)(2.1,0.6565625)(2.9524715,0.6530842)(3.804943,0.649606)(4.3476424,-0.54691577)(4.640456,-0.96430707)(4.93327,-1.3816984)(6.0011787,-2.7034376)(7.1379848,-2.668655)(8.274791,-2.6338723)(9.78,0.0565625)(10.02,0.8565625)
\rput(4.5628123,-1.6784375){$y_1$}
\psdots[dotsize=0.2](4.82,-1.2034374)
\rput(9.322812,-1.6184375){$y_2$}
\psdots[dotsize=0.2](8.9,-1.2234375)
\rput(7.14,-0.8){$y_{+}$}
\psdots[dotsize=0.2](7.05,-2.6634376)
\psdots[dotsize=0.2](7.05,-1.21)
\psline[linewidth=0.03cm,linestyle=dashed,dash=0.1cm 0.1cm](7.05,-1.2234375)(7.05,-2.6834376)
\rput(0.7428125,-0.7384375){$y_0$}
\psdots[dotsize=0.18](1.10,-1.2034374)
\rput(4.7,0.6715625){$G(0)=x$}
\psdots[dotsize=0.2](3.72,0.3165625)
\end{pspicture}}
\end{center}
\caption{The graph of $G(y)$ in \eqref{q:G}: $y_{+}$ is a local minimum.}
\label{fig1}
\end{figure}
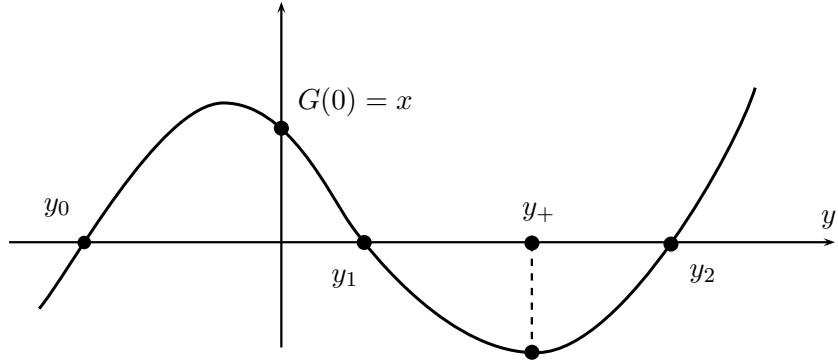

Under the assumption \eqref{q:dtf1} below on the timestep $k$,
the graph of the cubic $G$ is (qualitatively) as shown in Figure~\ref{fig1}.
We note in particular that $G(y)=0$ has a negative root $y_0$ and
two positive roots $y_1$ and $y_2$.
To verify this, we note the following.
First, $G(y)\to\pm\infty$ as $y\to\pm\infty$.
Next, $G(y)$ has two local extrema,
\begin{equation}
   y_\pm = \pm\Bigl(\frac{{\nu^3}}{3c_4 k}\Bigl[1+\frac{\nu k}{2\cpoi}\Bigr]\Bigr)^{1/2},
\end{equation}
with $y_-<0$ being a local maximum and $y_+>0$ a local minimum,
as verified by computing $G''(y_\pm)$.
Since $G(0)=x>0$ (the problem is trivial if $x=0$), we have $G(y_-)>0$.
Finally, computing
\begin{equation}
   G(y_+) = -\frac23\Bigl(1+\frac{\nu k}{2 \cpoi}\Bigr)\Bigl(\frac{{\nu^3}}{3c_4k}\Bigl[1+\frac{\nu k}{2 \cpoi}\Bigr]\Bigr)^{1/2} + x,
\end{equation}
we conclude that $G(y_+)<0$ if (this is essentially a restriction on $k$)
\begin{equation}\label{q:dtf1}
   |\gb u^{n-1}|^2 + \frac{2k}\nu|f|_{L^{\infty}(L^2)}^2
	< \frac23\Bigl(\frac{{\nu^3}}{3c_4k}\Bigr)^{1/2}.
\end{equation}
This implies the existence of the two positive roots $y_1$ and $y_2$
with $y_1<y_+<y_2$.

Now \eqref{q:un6} implies that $|\gb u^n|^2=y$ lies in the disjoint
set $[0,y_1]\cup[y_2,\infty)$.
However, $y_2>y_+\sim k^{-1/2}$, which is absurd for small $k$.
To prove that $y\not\in[y_2,\infty)$, we multiply $\eqref{q:full}_1$
by $2k(u^n-u^{n-1})$ in $L^2$ to get
\begin{equation}\label{q:est_err}\begin{aligned}
   2|u^n&-u^{n-1}|^2 + \nu k|\gb u^n|^2 - \nu k|\gb u^{n-1}|^2 + \nu k|\gb(u^n-u^{n-1})|^2\\
	&\quad= -2k (u^n\ncdot\gb u^n,u^n-u^{n-1}) + 2k (f^n,u^n-u^{n-1})
	=: I_1 + I_2.
\end{aligned}\end{equation}
Bounding the rhs as
\begin{align*}
   |I_2| &\le \frac{k}\nu|f^n|_{H^{-1}}^2 + {\nu k}|\gb(u^n-u^{n-1})|^2 \hskip-50pt\\
   |I_1| &= 2k\bigl|(u^n\ncdot\gb u^n,u^{n-1})\bigr|
	& &\le 2k |u^n|_{L^3}^{}|\gb u^n|_{L^2}^{}|u^{n-1}|_{L^6}^{}\\
	&\le ck |u^n|_{H^{1/2}}^{}|\gb u^n|\,|\gb u^{n-1}|
	& &\le ck\,|u^n|^{1/2}|\gb u^n|^{3/2}|\gb u^{n-1}|\\
	&\le \frac{\nu k}2|\gb u^n|^2 + \frac{ck}{\nu^3} |u^n|^2|\gb u^{n-1}|^4,
\end{align*}
and dropping the $2|u^n-u^{n-1}|^2$ on the lhs in \eqref{q:est_err}, we arrive at
\begin{equation}
   |\gb u^n|^2 \le \Bigl(2 + \frac{2 c_5}{{\nu^4}}|u^n|^2|\gb u^{n-1}|^2\Bigr)|\gb u^{n-1}|^2 + \frac2{\nu^2}|f|^2_{L^{\infty}(H^{-1})}.
\end{equation}
If we now assume that (effectively a timestep restriction)
\begin{equation}\label{q:dtf2}
   \Bigl(2 + \frac{2 c_5}{{\nu^4}}|u^n|^2|\gb u^{n-1}|^2\Bigr)|\gb u^{n-1}|^2 + \frac2{\nu^2}|f|^2_{L^{\infty}(H^{-1})}
	\le \Bigl(\frac{{\nu^3}}{3c_4k}\Bigr)^{1/2},
\end{equation}
noting that the rhs $<y_+<y_2$,
we can conclude that $|\gb u^n|^2 < y_2$ and therefore $|\gb u^n|^2 \in [0,y_1]$.
This gives us the local $H^1$ integrability of the scheme \eqref{q:full}:
if $k$ (is small enough that it) satisfies \eqref{q:dtf1} and \eqref{q:dtf2},
the one-step solution of \eqref{q:full} is bounded in $H^1$.
\end{proof}

\hbox to\hsize{\qquad\hrulefill\qquad}\medskip

\begin{proof}[Proof of Theorem~\ref{t:full}]
We begin with the short-time case and
assume the hypotheses (local in $n$) of Lemma~\ref{t:lem}.
Since $y_1$ is the root of a cubic $G(y_1)=0$, the bound $|\gb u^n|^2\le y_1$
is not very convenient, so we compute a more useful bound.
Recalling that $x>0$, we consider for some $a>0$
\begin{equation}\begin{aligned}
   G\bigl((1+ak)x;x\bigr)
	&= xk\,\Bigl[\frac{c_4}{\nu^3}(1+ak)^3x^2 - \Bigl(\frac{\nu}{2 \cpoi} + a\Bigr) - \frac{a\nu k}{2 \cpoi}\Bigr]\\
	&< xk\,\Bigl[\frac{c_4}{\nu^3}(1+ak)^3x^2 - a\Bigr].
\end{aligned}\end{equation}
Setting $a=2c_4x^2/\nu^3$, this implies $G\bigl((1+ak)x\bigr)<0$ if
\begin{equation}\label{q:dtf3}
   1 + ak \le 2^{1/3}
   \quad\Leftrightarrow\quad
   (|\gb u^{n-1}|^2 + 2k|f|_{L^{\infty}(L^2)}^2 /\nu)^2 2c_4k/\nu^3 \le 2^{1/3}-1.
\end{equation}
Assuming this, Lemma~\ref{t:lem} then gives us the explicit one-step estimate
\begin{align}
   \!\!\!\!\!\!\!\!\!\!\!\!|\gb u^n|^2 &\le y_1
	\le (1+ak)\,x\notag\\
	&= |\nabla u^{n-1}|^2
	  + \frac{2k}{\nu}|f|_{L^{\infty}(L^2)}^2
	  + \frac{2c_4 k}{\nu^3} \bigl(|\nabla u^{n-1}|^2 + ({2k}/{\nu}) |f|_{L^{\infty}(L^2)}^2 \bigr)^3,\label{q:e02}
\end{align}
which we can rewrite as
\begin{equation}
   \!\!\frac{|\gb u^n|^2 - |\gb u^{n-1}|^2}{k}
	\le \frac{2c_4}{\nu^3}\Bigl[\Bigl(|\gb u^{n-1}|^2 + \frac{2k}{\nu}|f|_{L^{\infty}(L^2)}^2\Bigr)^3 + \frac{\nu^2}{c_4}|f|_{L^{\infty}(L^2)}^2 \Bigr].
\end{equation}

To obtain a finite-time bound on $|\gb u^n|$, we proceed in analogy
with \eqref{q:dzdt} and define
\begin{equation}
   z_n := |\gb u^n|^2 + F
   \quad\textrm{where }
   F^3 = \frac{2\nu^2}{c_4}|f|_{L^{\infty}(L^2)}^2.
\end{equation}
By expanding both sides, we have
\begin{equation}
   \Bigl(|\gb u^{n-1}|^2 + \frac{2k}{\nu}|f|_{L^{\infty}(L^2)}^2\Bigr)^3 + \frac{\nu^2}{c_4}|f|_{L^{\infty}(L^2)}^2
	\le z_{n-1}^3,
\end{equation}
subject to the timestep restriction
\begin{equation}\label{q:dtf4}
   k \le  \frac{ {\nu^{5/3}}}{2c_4^{1/3}|f|_{L^{\infty}(L^2)}^{4/3}}
   \quad\Rightarrow\quad
   \begin{cases}
   &4^{1/3}c_4^{1/3}|f|_{L^{\infty}(L^2)}^{4/3}\,k \le \nu^{5/3},\\
   &4^{2/3}c_4^{2/3}|f|_{L^{\infty}(L^2)}^{8/3}\,k^2 \le {\nu^{10/3}}\vphantom{\Big|^|},\\
   &8c_4|f|_{L^{\infty}(L^2)}^4\,k^3 \le \nu^5.
   \end{cases}
\end{equation}
Then \eqref{q:e02} implies
\begin{equation}
   \frac{z_n-z_{n-1}}{k} \le \frac{2c_4}{\nu^3}z_{n-1}^3
	=: g(z_{n-1}).
\end{equation}
Arguing as we did in the semi-implicit case [cf.~\eqref{q:dd0}--\eqref{q:dd1}],
we conclude that $z_n\le 2z_0$ for all $n\ge0$ such that
$nk = t_n\le \nu^3/(8c_4\zeta_0^2) $.

This proves the theorem subject to the timestep restrictions,
which we collect here.
First, \eqref{q:dtfx} and \eqref{q:dtfy} are implied by
\begin{align}
   &\ktil \le \frac12\Bigl(\frac{\nu^3}{3c_4k}\Bigr)^{1/2}\label{q:dtfx1},\\
   &\Bigl(1 + \frac{c_5}{\nu^4} \ktil_0\ktil\Bigr)\ktil + |f|^2_{L^{\infty}(H^{-1})}/\nu^2 \le \Bigl(\frac{\nu^3}{12c_4k}\Bigr)^{1/2}\label{q:dtfy1},
\end{align}
where unlike in Lemma~\ref{t:lem}, here
$\ktil:=2\,|\gb u^0|^2 + 2\,(\nu^2|f|_{L^{\infty}(L^2)}^2/c_4)^{1/3} + ({10}c_0/\nu)|f|_{L^{\infty}(L^2)}^2$.
Next, \eqref{q:dtf4} is good as it stands.
Finally, using \eqref{q:dtf4} to handle the $k$ inside the bracket,
\eqref{q:dtf3} is implied by
\begin{equation}\label{q:dtfz}
   \biggl(2\,|\gb u^0|^2 + \frac{ {(1 + 2^{1/3})}  {\nu^{2/3}}  |f|_{L^{\infty}(L^2)}^{2/3}  }{c_4^{1/3}}\biggr)^2
	\le \frac{(2^{1/3}-1)\nu^3}{2c_4k}.
\end{equation}
This proves the short-time case.


\medskip
For small solutions, we first derive a more useful
explicit bound for $|\gb u^{n-1}|^2$.
We claim that with the assumption \eqref{q:hypf}, $|\gb u^n|^2\le y_1$
implies
\begin{equation}\label{q:linf}
   \Bigl(1+\frac{\nu k}{{4}\cpoi}\Bigr)|\gb u^n|^2
	\le |\gb u^{n-1}|^2 + \frac{2k}\nu |f|_{L^{\infty}(L^2)}^2.
\end{equation}
To prove this, we set
$y_*:=\bigl(|\gb u^{n-1}|^2 + 2k|f|_{L^{\infty}(L^2)}^2/\nu\bigr)/\bigl(1+\nu k/({4}\cpoi)\bigr)$
and compute
\begin{equation}\begin{aligned}
   G(y_*) = y_*\,\Bigl(1+\frac{\nu k}{{4}\cpoi}\Bigr)^{-2}\Bigl\{-\frac{\nu k}{{4}\cpoi}\Bigl(1+\frac{\nu k}{{4}\cpoi}\Bigr)^2 + \frac{c_4k}{\nu^3}\,x^2\Bigr\}.
\end{aligned}\end{equation}
Now $G(y_*)\le 0$ implies that $y_*\ge y_1$, and the former is true if
\begin{equation}\label{q:hypf0}
   |\gb u^{n-1}|^2 + \frac{2k}\nu|f|_{L^{\infty}(L^2)}^2
	= x \le \frac{\nu^2}{{2 \sqrt{ \cpoi c_4}}}.
\end{equation}

To obtain the uniform bound, we sum \eqref{q:linf} using \eqref{q:dgrone}
to find
\begin{equation}
   |\gb u^n|^2 \le \Bigl(1+\frac{\nu k}{{4}\cpoi}\Bigr)^{-n}|\gb u^0|^2
	+ \frac{{8}\cpoi}{\nu^2}|f|_{L^{\infty}(L^2)}^2 + \frac{2k}\nu|f|_{L^{\infty}(L^2)}^2.
\end{equation}
Assuming that
\begin{equation}\label{q:dtf0}
   k \le \cpoi/\nu,
\end{equation}
we can absorb the last term into the penultimate one to obtain \eqref{q:bdfh1}.
Consolidating our assumptions, the smallness condition \eqref{q:hypf0}
is now implied by \eqref{q:hypf},
while the timestep restrictions \eqref{q:dtf1} and \eqref{q:dtf2} can
both be satisfied by taking $k$ sufficiently small to satisfy
\begin{align}
   &\ktil_1 \le \frac12\Bigl(\frac{\nu^3}{3c_4k}\Bigr)^{1/2} \label{q:dtfa},\\
   &\Bigl(1 + \frac{c_5}{\nu^4} \ktil_0\ktil_1\Bigr)\ktil_1 + |f|^2_{L^{\infty}(H^{-1})}/\nu^2 \le \Bigl(\frac{\nu^3}{12c_4k}\Bigr)^{1/2}.\label{q:dtfb}
\end{align}
This proves the theorem.
\end{proof}


\vspace{10mm}
\end{document}